\documentclass[11pt,twoside]{article}
\usepackage{amsmath,amssymb,amsthm,euscript,color}
\usepackage[unicode,breaklinks=true,colorlinks=true]{hyperref}

\numberwithin{equation}{section}
\newtheorem{theorem}{Theorem}[section]
\newtheorem{corollary}[theorem]{Corollary}
\newtheorem{lemma}[theorem]{Lemma}

\newtheorem{assumption}[theorem]{Assumption}
\newcommand{\bke}[1]{\left ( #1 \right )}

\newcommand{\bket}[1]{\left \{ #1 \right \}}
\newcommand{\norm}[1]{ \| #1  \|}

\newcommand{\bka}[1]{{\left\langle #1 \right\rangle}}

\newcommand{\abs}[1]{\left | #1 \right |}

\def\al{\alpha}

\def\bfeta{\mathbf\eta}

\def\de{\delta}

\def\e {\varepsilon}

\def\la{\lambda}
\def\si{\sigma}

\def\wV{\widehat V}
\def\wP{\widehat P}
\def\Ga{\Gamma}
\def\De{\Delta}

\def\Si{\Sigma}

\def\Om{\Omega}

\newcommand{\w}{{\mathbf v}}
\newcommand{\loc} {\mathop{\mathrm{loc}}}
\newcommand{\R}{\mathbb{R}}
\def\Ha{{\mathfrak{H}}}

\newcommand{\Z}{\mathbb{Z}}
\newcommand{\N}{\mathbb{N}}

\renewcommand{\div}{\mathop{\rm div}}

\newcommand{\pd}{\partial}
\newcommand{\nb}{\nabla}
\newcommand{\td}{\tilde}

\newcommand{\wbar}[1]{\overline{\rule{0pt}{2.4mm} {#1}}}
\newcommand{\lec}{{\ \lesssim \ }}

\newcommand{\cC}{\mathcal{C}}

\newcommand{\I}{\infty}

\newcommand{\LL}{local-Leray }

\newcommand{\SVERAK}{{\v Sver\'ak}}

\newcommand{\donothing}[1]{}

\begin{document}

\title{Forward Self-Similar Solutions of the Navier-Stokes Equations
  in the Half Space}
\author{Mikhail Korobkov \and Tai-Peng Tsai}
\date{} \maketitle

{\bf Abstract}. For the incompressible Navier-Stokes equations in
the 3D half space, we show the existence of forward self-similar
solutions for arbitrarily large self-similar initial data.

\section{Introduction}
Let $\R^3_+=\{x=(x_1,x_2,x_3):x_3>0\}$ be a half space with
boundary $ \pd \R^3_+=\{x=(x_1,x_2,0)\}$.  Consider the 3D
incompressible Navier-Stokes equations for velocity $u:\R^3_+
\times [0,\I)\to \R^3$
  and pressure $p: \R^3_+ \times [0,\I)\to \R$,
\begin{equation}
\label{NS1} \pd _t u -\De u + (u \cdot \nb) u  + \nb p =0 , \quad
\div u = 0,
\end{equation}
in $ \R^3_+ \times [0,\I)$, coupled with the boundary condition
\begin{equation}
\label{NS3} u|_{\pd \R^3_+}  =0,
\end{equation}
and the  initial condition
\begin{equation}
\label{NS2} u|_{t=0}  = a, \quad \div a=0,\quad a|_{\pd \R^3_+}  =0.
\end{equation}

The system \eqref{NS1} enjoys a scaling property: If $u(x,t)$ is a
solution, then so is
\begin{equation}
\label{scaling}
u^{(\la)}(x,t):= \la u(\la x, \la^2 t)
\end{equation}
for any $\la>0$. We say that $u(x,t)$ is {\bf self-similar} (SS)
if $u=u^{(\la)}$ for every $\la>0$. In that case,
\begin{equation}
\label{SS.formula}
u(x,t)= \frac 1{\sqrt{2t}}\, U \bke{\frac x{\sqrt{2t}}},
\end{equation}
where $U(x) = u(x,\frac 1{2})$.  It is called {\bf discretely
  self-similar} (DSS) if $u=u^{(\la)}$ for one particular $\la>1$. To
get self-similar solutions $u(x,t)$ we usually assume the initial data
$a(x)$ is also self-similar, i.e.,
\begin{equation}
\label{minus1}
a(x)=\frac {a(\hat x)}{|x|},\quad \hat x=\frac x{|x|}.
\end{equation}

In view of the above, it is natural to look for solutions satisfying
\begin{equation}
\label{eq1.8}
|u(x,t)| \le \frac {C(C_*)}{|x|}, \quad \text{or }\quad
\norm{u(\cdot,t)}_{L^{3,\I}}\le C(C_*),
\end{equation}
where $C_*$ is some norm of the initial data $a$.  By $L^{q,r}$,
$1\le q,r\le \I$, we denote the Lorentz spaces.  In such classes,
with sufficiently small $C_*$, the unique existence of mild
solutions -- solutions of the integral equation version of
\eqref{NS1}--\eqref{NS2} via contraction mapping argument, -- has
been obtained by Giga-Miyakawa \cite{GM} and refined by {  Kato \cite{Kato},}
Cannone-Meyer-Planchon \cite{CMP, CP}, {  and Barraza \cite{Barraza}.}  It is also obtained in the
broader class BMO$^{-1}$ in Koch-Tataru \cite{Koch-Tataru}. In the
{ context} of the half space (and smooth exterior
domains), it follows from Yamazaki \cite{Yamazaki}. As a
consequence, if $a(x)$ is SS or DSS with small norm $C_*$ and
$u(x,t)$ is a corresponding solution satisfying \eqref{eq1.8} with
small $C(C_*)$, the uniqueness property ensures that $u(x,t)$ is
also SS or DSS, because $u^{(\la)}$ is another solution with
the~same bound and same initial data $a^{(\la)}=a$.  For large
$C_*$, mild solutions still make sense but there is no existence
theory since perturbative methods like the contraction mapping no
longer work.

Alternatively, one may try to extend the concept of weak solutions
(which requires $u_0 \in L^2(\R^3)$) to more general initial data.
One such theory is \LL solutions in $L^2_{uloc}$, constructed by
Lemari\'e-Rieusset \cite{LR}. However, there is no uniqueness
theorem for them and hence the existence of large SS or DSS
solutions was unknown.  Recently, Jia and \SVERAK{} \cite{JS}
constructed SS solutions for every SS $u_0$ which is locally
H\"older continuous. Their main tool is a local H\"older estimate
for \LL solutions near $t=0$, assuming minimal control of the
initial data in the large. This estimate enables them to prove a
priori estimates of SS solutions, and then to show their existence
by the Leray-Schauder degree theorem. This result is extended by
Tsai \cite{Tsai14} to the existence of discretely self-similar
solutions.

When the domain is the half space $\R^3_+$, however, there is so far
no analog theory of \LL solutions.  Hence the method of \cite{JS},
\cite{Tsai14} is not applicable.

In this note, our goal is to construct SS solutions in half space
for arbitrary large data. By $BC_w$ we denote bounded and weak-*
continuous functions. Our main theorem is the following.

\begin{theorem}
\label{th1.1} Let $\Om=\R^3_+$ and $A$ be the Stokes operator in
$\Om$ {{\rm(}see
{\rm(}\ref{Eq5.5}{\rm)}--{\rm(}\ref{notat}{\rm)}\,{\rm)}}. For any
self-similar vector field $a\in C^1_{loc}( \bar \Om \backslash
\{0\})$ satisfying $\div a=0$, $a|_{\pd \Om}=0$, there is a smooth
self-similar mild solution $u \in BC_w([0,\I);L^{3,\I}_\si(\Om))$
of \eqref{NS1}
  with $u(0)=a$ and
\begin{equation}
\label{Eq1.9}
\norm{u(t)-e^{-tA}a}_{L^2(\Om)}=C t^{1/4},\quad
\norm{\nb(u(t)-e^{-tA}a)}_{L^2(\Om)}=Ct^{-1/4},\quad \forall t>0.
\end{equation}
\end{theorem}

{\it Comments on Theorem \ref{th1.1}}
\begin{enumerate}
\item There is no restriction on the size of $a$.

\item It is concerned only with existence. There is no assertion on uniqueness.

\item Our approach also gives a second construction of large self-similar
  solutions in the whole space $\R^3$, but for initial data more
  restrictive ($C^1$) than those of \cite{JS}. In
  fact, it would show the existence of self-similar
  solutions in the cones
\[
K_\al = \bket{0 \le \phi \le \al}, \quad (0 < \al \le \pi),
\]
(in spherical coordinates), if one could verify Assumption
\ref{th3.1} for $e^{-\frac12 A}a$. We are able to verify it only
for $\al = \pi/2 $ and $\al=\pi$.

\item We have the uniform bound \eqref{eq1.8} for $u_0(t)=e^{-tA}a$
  and we will show $|u_0(x,t)| \lec (\sqrt t + |x|)^{-1}$ in Section
  \ref{S6}. We expect $u_0(t) \not \in L^q(\Om)$ for any $q\le 3$, and
  $\norm{u_0(t)}_{L^q} \to \I$ as $ t \to 0_+$ for $q>3$.  The
  difference $v=u-u_0$ is more localized: by interpolating
  \eqref{Eq1.9}, $\norm{v(t)}_{L^q} \to 0$ as $t \to 0_+$ for all $q
  \in [2,3)$. Although $\norm{v(t)}_{L^3(\Om)}=C$ for $t>0$, $v(t)$
    weakly converges to $0$ in $L^3$ as $t \to 0_+$, as easily shown
    by approximating the test function by $L^2 \cap L^{3/2}$
    functions.  Both $u_0(t)$ and $v(t)$ belong to
    $L^\I(\R_+;L^{3,\I}(\R^3_+))$.

\end{enumerate}

\medskip

We now outline our proof.  Unlike previous approaches based on the
evolution equations, we directly prove the existence of the
profile $U$ in \eqref{SS.formula}. It is based on the a priori
estimates for $U$ using the classical Leray-Schauder fixed point
theorem and the Leray \textit{reductio ad absurdum} argument
(which has been fruitfully applied in recent papers of Korobkov,
Pileckas and Russo \cite{KPR1}--\cite{KPR4} on the boundary value
problem of stationary Navier-Stokes equations). Specifically, the
profile $U(x)$ satisfies the Leray equations
\begin{equation}
\label{eq1.6}
-\De U - U - x \cdot \nb U + (U \cdot \nb) U  + \nb P =0 , \quad
\div U = 0
\end{equation}
in $\R^3_+$
with zero boundary condition and, in a suitable sense,
\begin{equation}
 U(x)\to U_0(x):= (e^{-\frac 12 A}a)(x)\quad\text{as }|x|\to\I.
\end{equation}
System \eqref{eq1.6} was proposed by Leray \cite{Leray34}, with
the~opposite sign for $U+x\cdot \nb U$, for the study of singular
{\it
  backward} self-similar solutions of \eqref{NS1} in $\R^3$ of the
form $u(x,t) = \frac 1{\sqrt{-2t}}\, U \bke{\frac
  x{\sqrt{-2t}}}$. Their triviality was first established in \cite{NRS} if $U \in L^3(\R^3)$, in particular if $U \in H^1(\R^3)$ as assumed in \cite{Leray34}, and then extended  in \cite{Tsai98} to $U \in L^q(\R^3)$, $3 \le q
\le \I$. In the
forward case and in the whole space setting, we have (see \cite{JS,
  Tsai14})
{ \begin{equation} \label{Eq1.14} |U_0(x)| \sim
|x|^{-1}, \quad V(x):= U(x) - U_0(x), \quad |V(x)| \lec |x|^{-2}\quad
(|x|>1).
\end{equation}}
In the half space setting, it is not clear if one can show
pointwise decay bound for $V$. We will however show that $V(x)$ is
a priori bounded in $H^1_0(\R^3_+)$, and use this a priori bound
to construct a solution. Due to lack of compactness of $H^1_0$ at
spatial infinity, we will use the {\it invading method},
introduced by Leray \cite{Leray33}:
We will approximate $\Om = \R^3_+$ by $\Om_k=\Om
\cap B_k$, $k =1,2,3,\ldots$, where $B_k$ is an increasing sequence of
concentric balls, construct solutions $V_k$ in $\Om_k$ of the
difference equation \eqref{eq3.4} with zero boundary condition, and
extract a subsequence converging to a desired solution $V$ in
$\R^3_+$.

Our proof is structured as follows.  We will first recall some
properties for Euler flows  in Section \ref{S2}, and then use it
to show that $V_k$ are uniformly bounded in $H^1_0(\Om_k)$ in
Section \ref{S3}. In Section \ref{S4}, we construct $V_k$ using
the a priori bound and a linear version of the Leray-Schauder
theorem, and extract a weak limit $V$ using the uniform bound. The
arguments in Sections \ref{S2}--\ref{S4} are valid as long as one
can show that $U_0=e^{-\frac
  12A_\Om}a$, $A_\Om$ being the Stokes operator in $\Om$, satisfies
certain decay properties to be specified in Assumption \ref{th3.1}. In
Sections \ref{S:BC} we show that, for $\Om=\R^3_+$ and those initial
data $a$ considered in Theorem \ref{th1.1}, $U_0$ indeed satisfies
Assumption \ref{th3.1}.  We finally verify that $u(x,t)$ defined by
\eqref{SS.formula} satisfies the assertions of Theorem \ref{th1.1} in
Section \ref{S6}.

Because our existence proof does not use the evolution equation, we do
not need the nonlinear version of the Leray-Schauder theorem as in
\cite{JS, Tsai14}.  As a side benefit, we do not need to check the
small-large uniqueness (cf.~\cite[Lemma 4.1]{Tsai14}).

\section{Some properties of solutions to the Euler system}
\label{S2}

{  For $q\ge1$ denote by $D^{1,q}(\Omega)$ the set of
functions $f\in W^{1,q}_{\loc}(\Omega)$ such that
$\|f\|_{D^{1,q}}(\Omega)=\|\nabla f\|_{L^q(\Omega)}<\infty$.
Recall, that by Sobolev Embedding Theorem, if $q<n$, then for any
$f\in D^{1,q}(\R^n)$ there exists a constant $c\in\R$ such that
$f-c\in L^{p}(\R^n)$\ \ with $p=\frac{nq}{n-q}$.}
In particular,
\begin{equation} \label{SET}
f\in D^{1,2}(\R^3)\Rightarrow f-c\in L^6(\R^3);\qquad f\in
D^{1,3/2}(\R^3)\Rightarrow f-c\in L^3(\R^3).
\end{equation}
Further, denote by $D^{1,2}_0(\Omega)$  the closure of the set of
all smooth functions having compact supports in $\Omega$ with
respect to the norm $\|\,\cdot\,\|_{D^{1,2}(\Omega)}$, \, and
$H(\Omega)=\{{\bf v}\in D^{1,2}_0(\Omega):\, \div {\bf v}=0\}$. In
particular,
\begin{equation} \label{SET1}
H(\Omega)\hookrightarrow L^6(\Omega)
\end{equation}
{(recall, that by Sobolev inequality $\|f\|_{L^6(\R^3)}\le
C\|\nabla f\|_{L^2(\R^3)}$ holds for every function $f\in
C^\infty_c(\R^3)$ having compact support in $\R^3$, see
\cite[Theorem 4.31]{Adams}\,).}

Assume that the following conditions are fulfilled:

\medskip
{\bf (E)} Let $\Om$ be a~domain in~$\R^3$ with (possibly
unbounded) connected Lipschitz boundary $\Gamma=\partial\Omega$,
and the functions $\w\in H(\Omega)$ and $p\in
D^{1,3/2}(\Omega)\cap L^3(\Om)$ satisfy the Euler system
\begin{equation} \label{2.1}\left\{\begin{array}{rcl} \big({\bf
v}\cdot\nabla\big){\bf v}+\nabla p & = & 0 \qquad \ \ \
\hbox{\rm in }\;\;\Omega,\\[4pt]
\div{\bf v} & = & 0\qquad \ \ \ \hbox{\rm in }\;\;\Omega,
\\[4pt]
{\bf v} &  = & 0\ \
 \qquad\  \hbox{\rm on }\;\;\partial\Omega.
\end{array}\right.
\end{equation}

The next statement  was proved in \cite[Lemma 4]{KaPi1} and in
\cite[Theorem 2.2]{Amick} { (see also
\cite[Lemma~4]{Amirat}\,)}.

\begin{theorem}
\label{kmpTh2.3'} {\sl Let the conditions {\rm (E)} be fulfilled.
Then
\begin{equation} \label{bp2} \exists\, \widehat p_0\in\R:\quad
p(x)\equiv \widehat p_0\quad\mbox{for }\Ha^2-\mbox{almost all }
x\in \partial\Omega.\end{equation} }
\end{theorem}

Here and henceforth we denote by $\mathfrak{H}^m$ the
$m$-dimensional Hausdorff measure, i.e.,
$\mathfrak{H}^m(F)=\lim\limits_{t\to 0+}\mathfrak{H}^m_t(F)$,
where $\mathfrak{H}^m_t(F)=\inf\bigl\{\sum\limits_{i=1}^\infty
\bigl({\rm diam}F_i\bigr)^m:\, {\rm diam} F_i\leq t, F\subset
\bigcup\limits_{i=1}^\infty F_i\bigr\}$.

\section{A priori bound for Leray equations}
\label{S3}

Recall that
the profile $U(x)$ in \eqref{SS.formula} satisfies Leray equations \eqref{eq1.6}
with zero boundary condition and $U(x) \to U_0(x)$ at spatial infinity.
Decompose
\begin{equation}
U =U_0 + V,\quad {U_0=e^{-\frac 12 A}a}.
\end{equation}
{Because $a$ is self-similar, $u_0(\cdot,t)=e^{-tA}a$ is also self-similar, $ u_0(x,t)=\la u_0 (\la x, \la^2 t)$ for all $\la>0$. Differentiating in $\la$ and evaluating at $\la=1$ and $t=1/2$, we get
\begin{equation}
0=U_0 + x \cdot \nb U_0 + \pd_t u_0(x,\frac 12) = U_0 + x \cdot \nb U_0  +  \De U_0 - \nb P_0,
\end{equation}
for some scalar $P_0$. Thus,}
the difference $V(x)$ satisfies
\begin{equation}
\label{eq3.4} -\De V - V - x \cdot \nb V + \nb P =F_0+F_1(V) ,
\quad \div V = 0,
\end{equation}
{for some scalar $P$},
where
\begin{align}
\label{eq3.5} F_0 &=%
{  - U_0\cdot \nb U_0,}
\\
\label{eq3.6} F_1(V) &=-( U_0+V) \cdot \nb V -  V  \cdot \nb
 U_0,
\end{align}
and $V$ vanishes at the boundary and the  spatial infinity.

For a Sobolev function $f\in W^{1,2}(\Omega)$ put
\begin{equation}
\|f\|_{H^1(\Omega)}:=\bke{\int_\Om  |\nb f|^2 +
\frac12|f|^2}^{1/2}.
\end{equation}
{Denote by $H_0^1(\Omega)$  the closure of the set of all
smooth functions having compact supports in $\Omega$ with respect
to the norm $\|\,\cdot\,\|_{H^1(\Omega)}$, \, and}
$$H_{0,\sigma}^1(\Omega)=\{f\in
H_0^1(\Omega):\div f=0\}.$$
Note that $H^1_0(\Om)=\{f \in
W^{1,2}(\Omega): f|_{\pd \Om}=0,\ \|f\|_{H^1(\Omega)}<\infty\}$
for bounded Lipschitz domains.

 We assume the following.

\begin{assumption}[Boundary data at infinity]
\label{th3.1} {\sl Let $\Om$ be a domain in $\R^3$. The vector
field $U_0:\Om\to \R^3$   satisfies \ {$\div U_0=0$} \
and
\begin{equation}
\label{eq3.4-as1}\|U_0\|_{L^6(\Omega)}<\infty,\qquad \|\nabla
U_0\|_{L^2(\Omega)}<\infty.
\end{equation}
}
\end{assumption}

Note that from Assumption \ref{th3.1} {and
(\ref{eq3.5})} it follows, in particular, that
\begin{equation}
\label{eq3.4-asc} {\abs{\int_\Om F_0 \cdot \eta }\le
C},\qquad\ \abs{\int_\Om (\eta \cdot \nb )U_0 \cdot \eta } \le C
\end{equation}
for any $\eta \in H_{0,\sigma}^1(\Omega)$ with $\norm{\eta}_
{H_{0,\sigma}^1(\Omega)}\le 1$ (by virtue of the evident imbedding
$H_{0,\sigma}^1(\Omega)\hookrightarrow L^p$ for all $p\in[2,6]$).

If it is valid in $\Om$, it is also valid in any subdomain of
$\Om$ with the same constant $C$. We will show in \S \ref{S:BC}
that for $\Om = \R^3_+$ and $a$ satisfying \eqref{Eq5.1}, $U_0=
e^{-\frac 12A}a$ satisfies \eqref{Eq5.3} and hence Assumption
\ref{th3.1}. This is also true if $\Om = \R^3$ and $a$ is
self-similar, divergence free, and locally H\"older continuous.

\begin{theorem}[A priori estimate for bounded domain]
\label{th3.2}  Let $\Om$ be a bounded domain in~$\R^3$ with
connected Lipschitz boundary $\partial\Omega$, and assume
Assumption \ref{th3.1} for $U_0$. Then for any function $V\in
H^1_{0}(\Om)$ satisfying
\begin{equation}
\label{eq3.4-l} -\De V  + \nb P =\lambda\bigl(V + x \cdot \nb
V+F_0+F_1(V)\bigr) , \quad \div V = 0,
\end{equation}
with some~$\lambda\in[0,1]$, we have the a~priori bound
\[
\norm{V}_{H^1(\Om)}^2=\int_\Om \bke{ |\nb V|^2 + \frac12|V|^2} \le
C(U_0,\Om).
\]
\end{theorem}

{\it Remark}.\quad Note that $C(U_0,\Om)$ is independent of
$\lambda\in [0,1]$.

\begin{proof}
Let the assumptions of the Theorem be fulfilled. Suppose that
its~assertion is not true. Then there exists a~sequence of
numbers~$\lambda_k\in[0,1]$ and functions~$V_k\in H^1_0(\Om)$ such
that
\begin{equation}
\label{eq3.4-lk} -\De V_k - \lambda_kV_k - \lambda_kx \cdot \nb
V_k + \nb P_k =\lambda_k\bigl(F_0+F_1(V_k)\bigr) , \quad \div V_k
= 0,
\end{equation}
moreover,
\begin{equation}
\label{bnorm-inf}J_k^2:=\int_{\Om} |\nb V_k|^2\to\infty.
\end{equation}
Multiplying the equation~(\ref{eq3.4-lk}) by~$V_k$ and integrating
by parts in~$\Omega$, we obtain the identity
\begin{equation}
\label{binv-17} J_k^2+\frac{\lambda_k}2\int_\Om|V_k|^2=\lambda_k
 \int_{\Om}(F_0-V_k\cdot \nb U_0)V_k.
\end{equation}
Consider the normalized sequence of functions
\begin{equation}
\label{binv-18} \widehat V_k=\frac{1}{J_k}V_k, \quad \widehat
P_k=\frac{1}{\lambda_kJ^2_k}P_k
\end{equation}
Since $$\int_{\Om} |\nb \wV_k|^2\equiv1,$$ we could extract
a~subsequence still denoted by $\wV_k$, which converges weakly in
$W^{1,2}(\Omega)$ to some function $V\in H^1_0(\Omega)$, and
strongly in $L^3(\Om)$. Also we could assume without loss of
generality that~$\lambda_k\to\lambda_0\in[0,1]$.

Multiplying the identity~(\ref{binv-17}) by~$\frac1{J^2_k}$ and
taking a~limit as~$k\to\infty$, we have
\begin{equation}
\label{binv-17-} 1+\frac{\lambda_0}2\int_\Om|V|^2=
 -\lambda_0\int_{\Om}(V\cdot \nb U_0)V=\lambda_0\int_{\Om}(V\cdot \nb V)U_0.
\end{equation}
In particular, $\lambda_k$ is separated from zero for large~$k$.

Multiplying the equation~(\ref{eq3.4-lk})
by~$\frac1{\lambda_kJ_k^2}$, we see that the pairs~$(\wV_k,\wP_k)$
satisfy the~equation
\begin{equation}
\label{binv-19} \wV_k \cdot \nb \wV_k+\nb\wP_k=\frac1{J_k}\biggl(
\frac1{\lambda_k}\De \wV_k + \wV_k + x \cdot \nb \wV_k
+{ \frac1{J_k}}F_0-U_0\cdot \nb \wV_k - \wV_k \cdot \nb
 U_0\biggr).
\end{equation}

Take arbitrary function~$\bfeta\in C^\infty_{c,\si}(\Om)$.
Multiplying (\ref{binv-19}) by~$\bfeta$, integrating by parts and
taking a limit, we obtain finally
\begin{equation}
\label{binv-20} \int_{\Omega}\bigl(V \cdot \nb
V\bigr)\cdot\bfeta=0.
\end{equation}
Since~$\bfeta$ was arbitrary function
from~$C^\infty_{c,\si}(\Om)$, we see that~$V$ is a weak solution
to the Euler equation
\begin{equation}
\label{b2.1-inv}\left\{\begin{array}{rcl} \big(V\cdot\nabla\big){
V}+\nabla P & = & 0 \qquad \ \ \
\hbox{\rm in }\;\;\Omega,\\[4pt]
\div V & = & 0\qquad \ \ \ \hbox{\rm in }\;\;\Omega,
\\[4pt]
V &  = & 0\ \
 \qquad\  \hbox{\rm on }\;\;\partial\Omega,
\end{array}\right.
\end{equation}
with some~$P\in D^{1,3/2}(\Omega)\cap L^3(\Omega)$. By
Theorem~\ref{kmpTh2.3'}, there exists a constant~$\widehat
p_0\in\R$ such that $P(x)\equiv \widehat p_0$ on $\partial\Omega$.
Of course, we can assume without loss of generality that $\widehat
p_0=0$, i.e., $P(x)\equiv 0$ on $\partial\Omega$. Then by
(\ref{binv-17-}) and (\ref{b2.1-inv}${}_1$) we get
$$1+\frac{\lambda_0}2\int_\Om|V|^2=
 -\lambda_0\int_{\Om} U_0\cdot\nabla P=-\lambda_0\int_{\Om} \div(P\cdot
 U_0)=0.
$$
The obtained contradiction finishes the proof of the Theorem.
\end{proof}

\begin{theorem}[A priori bound for invading method]
\label{th3.2-im}  Let $\Om=\R^3_+$, and assume Assumption
\ref{th3.1} for $U_0$. Take a~sequence of
balls~$B_k=B(0,R_k)\subset\R^3$ with $R_k\to\infty$, and consider
half-balls ~$\Omega_k=\Omega\cap B_k$. Then for functions $V_k\in
H^1_{0}(\Om_k)$ satisfying
\begin{equation}
\label{eq3.4-li} -\De V_k - V_k - x \cdot \nb V_k + \nb P_k
=F_0+F_1(V_k), \quad \div V_k = 0,
\end{equation}
we have the a~priori bound
\[
\int_{\Om_k} \bke{ |\nb V_k|^2 + \frac12|V_k|^2} \le  C(U_0)
\]
where the constant~$C(U_0)$ is independent of~$k$.
\end{theorem}

\begin{proof}
Let the assumptions of the Theorem be fulfilled. Suppose that
its~assertion is not true. Then there exists a~sequence of
domains~$\Omega_{k}$ and a sequence of solutions $V_k \in H^1_0(\Om_k)$ of
\eqref{eq3.4-li}
such that \begin{equation}
\label{norm-inf}J_k^2:=\norm{V_k}_{H^1(\Om_k)}^2=\int_{\Om_k}
\bke{ |\nb V_k|^2 + \frac12|V_k|^2}\to\infty.
\end{equation}
Multiplying the equation~(\ref{eq3.4-li}) by~$V_k$ and integrating
by parts in~$\Omega_k$, we obtain the identity
\begin{equation}
\label{inv-17} J_k^2=
 \int_{\Om_k}(F_0-V_k\cdot \nb U_0)V_k.
\end{equation}
Consider the normalized sequence of functions
\begin{equation}
\label{inv-18} \widehat V_k=\frac{1}{J_k}V_k, \quad \widehat
P_k=\frac{1}{J^2_k}P_k
\end{equation}
Multiplying the equation~(\ref{eq3.4-li}) by~$\frac1{J_k^2}$, we
see that the pairs~$(\wV_k,\wP_k)$ satisfy the~equation
\begin{equation}
\label{inv-19} \wV_k \cdot \nb \wV_k+\nb\wP_k=\frac1{J_k}\bigl(
\De \wV_k + \wV_k + x \cdot \nb \wV_k +F_0-U_0\cdot \nb \wV_k -
\wV_k \cdot \nb
 U_0\bigr).
\end{equation}
Since $$\int_{\Om_k} \bke{ |\nb \wV_k|^2 +
\frac12|\wV_k|^2}\equiv1,$$ we could extract a~subsequence still
denoted by $\wV_k$, which converges weakly in $W^{1,2}(\Omega)$ to
some function $V\in H_0^1(\Omega)$, and strongly in $L^2(E)$ for
any $E \Subset \wbar \Om$.

Multiplying the identity (\ref{inv-17}) by~$\frac1{J^2_k}$ and
taking a~limit as~$k\to\infty$, we have
\begin{equation}
\label{inv-17-} 1=
 \int_{\Om}(-V\cdot \nb U_0)V.
\end{equation}
Take arbitrary function~$\bfeta\in C^\infty_{c,\si}(\Om)$.
Multiplying (\ref{inv-19}) by~$\bfeta$, integrating by parts and
taking a limit, we obtain finally
\begin{equation}
\label{inv-20} \int_{\Omega}\bigl(V \cdot \nb
V\bigr)\cdot\bfeta=0.
\end{equation}
Since~$\bfeta$ was arbitrary function
from~$C^\infty_{c,\si}(\Om)$, we see that~$V$ is a weak solution
to the Euler equation
\begin{equation}
\label{2.1-inv}\left\{\begin{array}{rcl} \big(V\cdot\nabla\big){
V}+\nabla P & = & 0 \qquad \ \ \
\hbox{\rm in }\;\;\Omega,\\[4pt]
\div V & = & 0\qquad \ \ \ \hbox{\rm in }\;\;\Omega,
\\[4pt]
V &  = & 0\ \
 \qquad\  \hbox{\rm on }\;\;\partial\Omega,
\end{array}\right.
\end{equation}
with some~$P\in D^{1,3/2}(\Omega)\cap L^3(\Omega)$. More
precisely, since $V,\nabla V\in L^2(\Omega)$, we have $P\in
D^{1,q}(\Omega)$ for every $q\in[1,3/2]$, consequently, $P\in
L^s(\Omega)$ for each $s\in[3/2,3]$. { In particular,
$P\in L^3(\Omega)$ \,and\, $\nabla P\in L^{\frac98}(\Omega)$,
furthermore,
\begin{align*}
\int_{S^+_R} |P|^{\frac43}
&= - R^2\int_R^\infty  \int_{S^+_1}\frac d{dr}
\bke{|P(r\omega)|^{\frac43}} d \omega\, dr
\\
&\lec \int_{|x|>R}
|P|^{\frac13} |\nabla
P|\le\biggl(\int_{|x|>R}|P|^3\biggr)^{\frac19}\biggl(\int_{|x|>R}|\nabla
P|^{\frac98}\biggr)^{\frac89},
\end{align*}
where $S_R^+=\{x\in\Omega:|x|=R\}$ is the corresponding
half-sphere. Hence we conclude that}
\begin{equation}
\label{2.1-invs-k1}\int\limits_{S^+_R}|P|^{4/3}\to0\qquad\mbox{ as
}R\to\infty.
\end{equation}
 Analogously, from the assumption
$U_0\in L^{6}(\Omega)$, $\nabla U\in L^2(\Omega)$ it is very easy
to deduce that
\begin{equation}
\label{2.1-invs-k2} \int\limits_{S^+_R}|U_0|^4\to0\qquad\mbox{ as
}R\to\infty.
\end{equation}
From the other hand, by (\ref{inv-17-}) and (\ref{2.1-inv}$_1$) we
obtain
\begin{equation}
\label{inv-17---} 1=
 \int_{\Om}(V\cdot \nb)V\cdot U_0 =
- \int_{\Om}\nabla P\cdot
U_0=-\lim_{R\to\infty}\int_{\Om_R}\div\bigl(P\cdot
U_0\bigr)=-\lim_{R\to\infty}\int_{S^+_R}P \bigl(U_0\cdot\mathbf
n\bigr)=0
\end{equation}
where $\Omega_R=\Omega\cap B(0,R)$ and the last equality follows
from (\ref{2.1-invs-k1})--(\ref{2.1-invs-k2}). The obtained
contradiction finishes the proof of the Theorem.
\end{proof}

\section{Existence for Leray equations}
\label{S4}

The proof of existence theorem for the system of
equations~\eqref{eq3.4}--\eqref{eq3.6}  in bounded domains~$\Om$
is based on the following fundamental fact.

\begin{theorem}[Leray--Schauder Theorem] \label{L-ShT} {\sl Let
$S:X\to X$  be a continuous and compact mapping of a Banach
space~$X$ into itself, such that the set
$$\{x\in X:x=\lambda Sx\quad\mbox{ for some }\lambda\in[0,1]\}$$
is bounded. Then~$S$ has a fixed point $x_*=Sx_*$.}
\end{theorem}

 Let $\Om$ be a~domain in~$\R^3$ with
 connected Lipschitz boundary $\Gamma=\partial\Omega$, and put
$X=H^1_{0,\si}(\Om)$.

For functions~$V_1,V_2\in H^1_{0,\si}(\Om)$
denote~$\bka{V_1,V_2}_H=\int_{\Omega}\nabla V_1\cdot\nabla V_2$.
Then the system \eqref{eq3.4}--\eqref{eq3.6} is equivalent to the
following identities:
\begin{equation}
\label{iLs} \bka{V,\zeta}_H = \int_\Omega G(V)\cdot\zeta,\quad
\forall \zeta \in C^\I_{c,\si}(\Om),
\end{equation}
where $G(V)=V+x\cdot\nb V+F(V)$, \ $F(V)=F_0+F_1(V)$,
\begin{align}
F_0(x)&=%
-U_0 \cdot \nb U_0,
\\
F_1(V) &=-(U_0+V) \cdot \nb V -  V  \cdot \nb U_0.
\end{align}
Since $H^1_{0,\si}(\Om)\hookrightarrow L^6(\Omega)$, by Riesz
representation theorem, for any $f\in L^{6/5}(\Omega)$ there
exists a~unique mapping~$T(f)\in H^1_{0,\si}(\Om)$ such that
\begin{equation}
\label{iLs1} \bka{T(f),\zeta}_H = \int_\Omega f\cdot\zeta,\quad
\forall \zeta \in C^\I_{c,\si}(\Om),
\end{equation}
moreover,
$$\|T(f)\|_H\le \|f\|_{X'},$$
where
$$\|f\|_{X'}=\sup\limits_{\zeta\in C^\I_{c,\si}(\Om),\
\|\zeta\|_H\le1}\int_\Omega f\cdot\zeta.$$ Then the system
\eqref{eq3.4}--\eqref{eq3.6}$\sim$\eqref{iLs} is equivalent to the
equality
\begin{equation}
\label{coV-eq} V=T(G(V)).
\end{equation}

\begin{theorem}[Compactness]
\label{bp-CompT} {\sl If  $\Om$ is a~bounded domain in~$\R^3$ with
connected Lipschitz boundary $\Gamma=\partial\Omega$, and
Assumption \ref{th3.1} holds for $U_0$, then for
$X=H^1_{0,\si}(\Om)$ the operator $S:X\ni V\mapsto T(G(V))\in X$
is continuous and compact.}
\end{theorem}

\begin{proof}
(i) For $V, \td V \in X$, denoting $v = \td V - V$,
\[
F(\td V) - F(V) = - (U_0+V+v)\cdot \nb v - v \cdot \nb (U_0+V).
\]
Thus we have
\begin{align}
\nonumber &\norm{S(\td V) - S(V)}_X
\\
\nonumber &\lec\|v\|_{L^2}+\|\nabla v\|_{L^2}+\norm{F(\td V) -
F(V)}_{L^2 + L^{6/5}}
\\
\nonumber &\lec \|v\|_{L^2}+\|\nabla v\|_{L^2}+\norm{U_0}_{L^2}
\norm{\nb v}_{L^2} + \norm{V+ v}_{L^2} \norm{\nb v}_{L^2} +
\norm{\nb U_0}_{L^2} \norm{ v}_{L^2} + \norm{v}_{L^2} \norm{\nb
V}_{L^2}
\\
&\lec (1+\norm{V}_X+\norm{v}_X) \norm{v}_X.
\end{align}

(ii) By Sobolev Theorems, we have the compact embedding:
$X\hookrightarrow L^r(\Omega) \quad \forall\,r\in[1,6)$. Thus if a
sequence $V_k\in X$ is bounded in~$X$, i.e.,
$\|V_k\|_{L^2(\Omega)}+\|\nabla V_k\|_{L^2(\Omega)}\le C$, then we
can extract a~subsequence $V_{k_l}$ which converges to some $V\in
X$ in $L^3(\Omega)$ norm: $\|V_{k_l}-V\|_{L^3(\Omega)}\to0$
as~$l\to\infty$. Then using the condition $V_{k_l}\equiv V\equiv0$
on~$\partial\Omega$ and integration by parts, it is easy to see
that~$\|F(V_{k_l})-F(V)\|_{X'}\to0$ and, consequently,
$\|G(V_{k_l})-G(V)\|_{X'}\to0$ as~$l\to\infty$.
\end{proof}

\begin{corollary}[Existence in bounded domains]
\label{EiBD} {\sl  Let $\Om$ be a~bounded domain in~$\R^3$ with
connected Lipschitz boundary $\partial\Omega$, and assume
Assumption \ref{th3.1} for $U_0$. Then the system
\eqref{eq3.4}--\eqref{eq3.6} has a~solution~$V\in
H^1_{0,\si}(\Om)$.}
\end{corollary}

\begin{proof}
This is a direct consequence of
Theorems~\ref{L-ShT}--\ref{bp-CompT} and~\ref{th3.2}.
\end{proof}

\begin{theorem}[Existence in unbounded domains]
\label{EiuBD} {\sl Let $\Om=\R^3_+$, and assume Assumption
\ref{th3.1} for $U_0$. Then the system
\eqref{eq3.4}--\eqref{eq3.6} has a~solution~$V\in
H^1_{0,\si}(\Om)$.}
\end{theorem}

\begin{proof}
 Take balls~$B_k=B(0,k)$ and
consider the~increasing sequence of domains~$\Omega_k=\Omega\cap
B_k$ from Theorem~\ref{th3.2-im}. By Corollary~\ref{EiBD} there
exists a~sequence of solutions~$V_k\in H^1_{0,\si}(\Om_k)$ of
the~system~\eqref{eq3.4}--\eqref{eq3.6} in $\Omega_k$. By
Theorem~\ref{th3.2-im}, the norms $\|V_k\|_ {H^1_{0,\si}(\Om)}$
are uniformly bounded, thus we can extract a~subsequence $V_{k_l}$
such that the~weak convergence~$V_{k_l}\rightharpoonup V$ in
$W^{1,2}(\Omega')$ holds for any bounded
subdomain~$\Omega'\subset\Omega$. It is easy to check that the
limit function~$V$ is a solution of
the~system~\eqref{eq3.4}--\eqref{eq3.6} in $\Omega$.
\end{proof}

\section{Boundary data at infinity in the half space}
\label{S:BC}

In this section we restrict ourselves to the half space $\Om=\R^3_+$
with boundary $\Si = \pd \R^3_+$ and study the decay property of $U_0
= e^{-\frac 12A}a$. Our goal is to prove the following lemma, which
ensures Assumption \ref{th3.1} under the conditions of Theorem
\ref{th1.1}.

Denote $x^* = (x',-x_3)$ for $x=(x',x_3)\in \R^3$, and $\bka{z}=
(1+|z|^2)^{1/2}$ for $z \in \R^m$.

\begin{lemma}
\label{th:U0}
Suppose $a$ is a vector field in $\Om=\R^3_+$ satisfying \begin{equation}
\label{Eq5.1}
\begin{split}
a\in C^1_{loc}( \bar \Om \backslash \{0\};\R^3), \quad
&\div a=0, \quad a|_{\pd \Om}=0, \\
 a(x) = \la a(\la x)\quad & \forall x\in \Om, \ \forall \la>0.
\end{split}
\end{equation}
Let $U_0 = e^{-\frac 12A}a,$ where $A$ is the Stokes operator in
$\Om$. Then
\begin{equation}
\label{Eq5.2} |\nb ^k U_0(x)| \le c_k[a]_1 (1+x_3)^{-\min(1,k)}
(1+| x|)^{-1} , \quad \forall k \in \Z_+=\{0,1,2,\dots\},
\end{equation}
and, for any $0<\de\ll 1$,
\begin{equation}
\label{Eq5.3}
|\nb U_0(x)|
\le c_\de [a]_1
x_3^{-\de}\bka{x}^{2\de-2},
\end{equation}
where $[a]_m=\sup_{k\le m, |x|=1} |\nb^k a(x)|$.

If we further assume $a \in C^m_{loc}$, $m \ge 2$, and $\pd_3^k a|_{\Si}=0$ for $k<m$, then $|\nb ^k
U_0(x)| \le c_k [a]_m \bka{x_3}^{-k} \bka{ x}^{-1} $ for $k \le m$.
 \end{lemma}

Estimates \eqref{Eq5.2} and \eqref{Eq5.3} imply, in particular,
\begin{equation}
U_0 \in L^4(\Om) \cap L^\infty(\Om), \quad \nb U_0 \in
L^2(\Om),
\end{equation}
and hence Assumption \ref{th3.1} for $U_0$ is satisfied.

\subsection{Green tensor for nonstationary Stokes system in half space}
Consider the nonstationary Stokes system in half space $\R^3_+$,
\begin{equation}
\label{Eq5.5}
\pd_t v - \De v + \nb p = 0, \quad \div v=0, \quad (x\in \R^3_+, \ t>0),
\end{equation}
\begin{equation}
\label{Eq5.6}
v|_{x_3=0}=0, \quad v|_{t=0}=a.
\end{equation}
{In our notation,
\begin{equation}
\label{notat} v(t)=e^{-tA}a.
\end{equation}}

It is shown by Solonnikov \cite[\S2]{MR1992567} that,
if $a=\breve a$ satisfies
\begin{equation}
\label{solenoidal}
\div \breve a = 0, \quad \breve a_3|_{x_3=0}=0,
\end{equation}
then
\begin{equation}
\label{eq5.3}
v_i(x,t) = \int_{\R^3_+} \breve G_{ij}(x,y,t) \breve a_j(y)dy
\end{equation}
with
\begin{align}
\breve G_{ij}(x,y,t)& = \de_{ij}\Ga(x-y,t) + G_{ij}^*(x,y,t) \label{Eq5.7}
\\
G_{ij}^*(x,y,t) &=  -\de_{ij}\Ga(x-y^*,t) \nonumber
 \\
&\quad - 4(1-\de_{j3})\frac {\pd}{\pd x_j}
 \int_{\R^2 \times [0,x_3]} \frac {\pd}{\pd x_i} E(x-z) \Ga(z-y^*,t)\,dz,
\nonumber
\end{align}
where $E(x) = \frac 1{4\pi |x|}$ and $\Ga(x,t)=(4\pi t)^{-3/2}
e^{-\frac {|x|^2}{4t}}$ are the fundamental solutions of the Laplace
and heat equations in $\R^3$.  (A sign difference occurs since $E(x) =
\frac {-1}{4\pi |x|}$ in \cite{MR1992567}.)  Moreover, $G_{ij}^*$
satisfies the pointwise bound (\cite[(2.38)]{MR1992567})
\begin{equation}
\label{Solonnikov.est}
|\pd_t^ s D_x ^k D_y^\ell G_{ij}^*(x,y,t)|\lec t^{-s - \ell_3/2}
(\sqrt t+x_3)^{-k_3} (\sqrt t+|x-y^*|)^{-3- |k'|- |\ell'|} e^{-\frac{c
    y_3^2}t}
\end{equation}
for all $s\in \N=\bket{0,1,2,\ldots}$ and $k,\ell \in \N^3$.

Note that $\breve G_{ij}$ is not the Green tensor in the strict sense
since it requires \eqref{solenoidal}.  There is no known pointwise
estimate for the Green tensor, cf.~Solonnikov \cite{MR0171094} and
Kang \cite{MR2097573}.

We now estimate $U_0 = e^{- \frac 12A}a$ for $a$ satisfying
\eqref{Eq5.1}.  By \eqref{eq5.3} and \eqref{Eq5.7},
\begin{equation}
\label{eq5.12}
U_{0,i}(x) = \int_{\R^3_+} \Ga(x-y,\tfrac12) a_i(y)dy + \int_{\R^3_+}
G_{ij}^*(x,y,\tfrac12) a_j(y)dy =: U_{1,i}(x)+ U_{2,i}(x).
\end{equation}
By \eqref{Solonnikov.est}, for $k \in\Z_+$ and using only
$|a(y)|\lec \frac 1{|y'|}$,
\begin{align}
\nonumber
|\nb ^k U_2(x)| &\lec \int_{\R^3_+} (1+x_3)^{-k} (1+x_3+|x'-y'|)^{-3} e^{- cy_3^2}\frac 1{|y'|} dy
\\
\nonumber
&\lec (1+x_3)^{-k}  \int_{\R^2} (1+x_3+|x'-y'|)^{-3} \frac 1{|y'|} dy'
\\
\nonumber
& = (1+x_3)^{-k-2}\int_{\R^2} (1+|\bar x-z'|)^{-3} \frac 1{|z'|} dz',\quad (\bar x = \frac{x'}{1+x_3})
\\
\nonumber
&\lec (1+x_3)^{-k-2} (1+|\bar x|)^{-1}
\\
\label{eq5.13}
&= (1+x_3)^{-k-1} (1+x_3+| x'|)^{-1}.
\end{align}
To estimate $U_1$, fix a cut-off function $\zeta(x)\in C^\I_c(\R^3)$
with $\zeta(x)=1$ for $|x|<1$. We have
\begin{equation}
\label{eq5.14}
\nb^k U_{1,i}(x) = \int_{\R^3_+} \Ga(x-y,\tfrac12)\nb^k_y
((1-\zeta(y)) a_i(y)) dy + \int_{\R^3_+} \nb^k_x\Ga(x-y,\tfrac12)
(\zeta(y) a_i(y)) dy
\end{equation}
where we used $a|_{\Si}=0$, and hence, for $k\le 1$,
\begin{equation}
\label{eq5.15}
|\nb^k U_1(x)|\lec \int_{\R^3} e^{-|x-y|^2/2} \bka{y}^{-1-k}dy + e^{-x^2/4} \lec
\bka{x}^{-1-k}.
\end{equation}
We can get the same estimate for $k \ge 2$ if we assume $\nb^k a$
is defined and has the same decay. On the other hand, we can show
$|\nb^k_xU_1(x)| \lec \bka{x}^{-2}$ for $k \ge 2$ if we place the
extra derivatives on $\Ga$ in the first integral of
\eqref{eq5.14}.

Combining \eqref{eq5.13} and \eqref{eq5.15}, we get \eqref{Eq5.2} and
the last statement of Lemma \ref{th:U0}.

Denote
\begin{equation}
\Om_- = \{x\in \Om: 1+x_3 > |x'|\}, \quad
\Om_+ = \{x\in \Om: 1+x_3 \le |x'|\}.
\end{equation}
By \eqref{eq5.13} and \eqref{eq5.15}, %
we have shown
\eqref{Eq5.3} in $\Om_-$ (with $\de=0$).

It remains to show \eqref{Eq5.3} in $\Om_+$.

\subsection{Estimates using boundary layer integral}

Denote $\e_j=1$ for $j<3$ and $\e_3=-1$.  Thus $x^*_j = \e_j x_j$.
Let $\bar a(x)$ be an extension of $a(x)$ to $x\in \R^3$ with
\[
\bar a_j(x) = \e_j a_j(x^*), \quad \text{if }x_3 <0.
\]
Since $\div a=0$ in $\R^3_+$ and $a|_{\Si}=0$, it follows that $\div
\bar a=0$ in $\R^3$.  Let $u(x,t)$ be the solution of the
nonstationary Stokes system in $\R^3$ with initial data $\bar a$,
given simply by
\[
u_i(x,t) = \int_{\R^3} \Ga(y,t) \bar a_i(x-y)\,dy.
\]
It follows that
$u_i (x,t) = \e_i u_i(x^*,t)$.
Thus
\begin{equation}
\pd_3 u_i(x,t)|_{\Si}=0, \quad (i<3);\quad
u_3(x,t)|_{\Si}=0.
\end{equation}
We have $|\nb^k a(y)| \lec |y|^{-1-k}$ for $k\le 1$.  By the same
estimates leading to \eqref{eq5.15} %
for $U_1$,
we have
\begin{equation}
\label{eq6.13}
|\nb^k_x u_i(x,\tfrac 12)|
\lec \bka{x}^{-1-\min(1,k)}, \quad (k \le 2).
\end{equation}
Thus $u(x,\tfrac 12)$ satisfies  \eqref{Eq5.3}.

Using self-similarity condition
\begin{equation}
\label{sc_u} u(x,t)= \la u(\la x, \la^2 t)\qquad \forall\la>0,
\end{equation}
from (\ref{eq6.13}) %
we get
\begin{equation}
\label{eq6.13t}
|\nb^m_x u_i(x,t)| \lec
\left \{
\begin{aligned}
\bigl(|x|+\sqrt{t}\bigr)^{-1-m}, \quad (m \le 1),
\\
t^{-1/2}\bigl(|x|+\sqrt{t}\bigr)^{-2}, \quad (m = 2).
\end{aligned}
\right.
\end{equation}

Decompose now
\[
v = u - w.
\]
Then $w$ satisfies the nonstationary Stokes system in $\R^3_+$ with
zero force, zero initial data, and has boundary value
\begin{equation}
\label{u:bc}
w_j(x,t)|_{x_3=0} = u_j(x',0,t), \quad \text{if }j<3; \quad
w_3(x,t)|_{x_3=0} =0.
\end{equation}
It is given by the boundary layer integral (using \eqref{u:bc}),
\begin{equation}
w_i(x,t) = \sum_{j=1,2}\int_0^t \int_\Si K_{ij}(x-z',s) u_j(z',0,t-s)\,dz'ds ,
\end{equation}
where, for $j<3$,
 (\cite[pp.~40, 48]{MR0171094})
\begin{align}
K_{ij}(x,t)&= -2\de_{ij}\pd_3 \Ga -\frac 1{\pi} \pd_j \mathcal{C}_i ,
\\
\label{Ci.def}
\cC_i(x,t) &= \int_{\Si \times [0,x_3]} \pd_3 \Ga(y,t) \frac {y_i -x_i} {|y-x|^3}\,dy.
\end{align}
(Note that $K_{i3}$ ($j=3$) have extra terms.)  They satisfy for $j<3$
(\cite[pp.~41, 48]{MR0171094})
\begin{align}
\label{eq6.16} |\pd_t^m D_{x'}^\ell \pd_{x_3}^k \cC_i(x,t)| &\le c
t^{-m-\frac12} (x_3+\sqrt t)^{-k} (|x|+\sqrt t)^{-2-\ell}.
\end{align}

We now show \eqref{Eq5.3} for $w(x,1/2)$ in the region $\Om_+: 1+x_3
\le |x'|$.

For $t=1/2$ and $i,k \in \{1,2,3\}$,
\begin{align}
\nonumber
\pd_{x_k} w_i(x,\tfrac12)& = - \sum_{j=1,2}\int_{0}^{\frac 12} \int_\Si  \frac 1{\pi} \pd_k \mathcal{C}_i (x-z',s)  \pd_{z_j} u_j(z',0,\tfrac 12-s)\,dz'ds
\\
\nonumber
&\quad
 -\mathbf{1}_{i<3}\int_0^{\frac 12} \int_\Si  2\pd_k \pd_3 \Ga  (x-z',s) u_i(z',0,\tfrac 12-s)\,dz'ds
\\
\label{EQ5.24}
&=I_1+I_2.
\end{align}
Above, we have integrated by parts in tangential directions $x_j$ in
$I_1$.

By \eqref{eq6.13t} and \eqref{eq6.16},
\[
|I_1|\lec \int_0^{\frac12} \int_\Si s^{-\frac12} (x_3+ \sqrt {s})^{-1}
(|x-z'|+\sqrt {s})^{-2}(|z'|+\sqrt {\tfrac12-s})^{-2}\,dz'\,ds.
\]
Fix $0<\e \le 1/2$.  Splitting $(0,1/2) = (0,1/4] \cup (1/4,1/2)$, and
changing variable $s \to 1/2 - s$ in $(1/4,1/2)$, we get
\begin{align*}
| I_1|\lec &  \int_0^{\frac14} \int_\Si x_3^{-2\e} s^{-1+\e}  (|x'-z'|+x_3+\sqrt {s})^{-2}(|z'|+1)^{-2}\,dz'\,ds
\\
& + \int_0^{\frac14} \int_\Si  (x_3+ 1)^{-1} (|x'-z'|+x_3+1)^{-2}(|z'|+\sqrt s)^{-2}\,dz'\,ds.
\end{align*}
Integrating first in time and using that, for $0<b<\I$, $0 \le a
<1<a+b$, and $0<N<\I$,
\begin{equation}
\label{Eq-Nab} \int_0^{1} \frac {ds}{s^a (N+s)^b} \le \frac
C{N^{a+b-1} (N+1) ^{1-a}},
\end{equation}
\begin{equation}
\int_0^{1} \frac {ds}{s^a (N+s)^{1-a}} \le C \min \bke{ \frac
  1{N^{1-a}}, \log \frac{2N+2}N} ,
\end{equation}
where the constant $C$ is independent of $N$, we get
\begin{align*}
| I_1|\lec &   \int_\Si x_3^{-2\e}  (|x'-z'|+x_3)^{-2+2\e} (|x'-z'|+x_3+1)^{-2\e}(|z'|+1)^{-2}\,dz'
\\
& +  \int_\Si  (x_3+ 1)^{-1} (|x'-z'|+x_3+1)^{-2}\min \bke{ \frac 1{|z'|^2}, \log \frac{2|z'|^2+2}{|z'|^2}} \,dz'.
\end{align*}
Dividing the integration domain to $|z'|< |x'|/2$, $|x'|/2<|z'|<2 |x'|$
and $|z'|> 2|x'|$, we get
\begin{equation}
| I_1|\lec  x_3^{-2\e} \bka{x}^{-2+\de},\quad (x\in \Om_+)
\end{equation}
for any $0<\de\ll 1$.  Taking $\e=\de/2$ and $\e=1/2$, we get
\begin{equation}
\label{Eq5.30}
(1+x_3) |I_1| \lec  x_3^{-\de}\bka{x}^{-2+2\de} ,\quad (x\in \Om_+).
\end{equation}

To estimate $I_2$ for $i<3$ (note $I_2 =0$ if $i=3$), we separate two
cases. If $k<3$, integration by parts gives
\[
I_2 = -\int_0^{\frac 12} \int_\Si 2 \pd_3 \Ga (x-z',s) \pd_{z_k}
u_i(z',0,\tfrac 12-s)\,dz'ds .
\]
Using $u e^{-u^2} \le C_\ell (1+u)^{-\ell}$ for $u>0$ for any
$\ell>0$,
\begin{equation}
\label{Eq5'1} \pd_3 \Ga(x,s) = cs^{-2} \frac{x_3}{\sqrt
s}e^{-x^2/4s} \le cs^{-2}(1+\frac{|x|}{\sqrt s})^{-3} =c s^{-1/2}
(|x|+\sqrt s)^{-3}.
\end{equation}
Hence $I_2$ can be estimated in the same way as $I_1$, and
\eqref{Eq5.30} is valid if $I_1$ is replaced by $I_2$ and $k<3$.

When $k=3$, by $\pd_t \Ga = \De \Ga$ and integration by parts,
\begin{align*}
I_2 &= \int_0^{\frac 12} \int_\Si  2 (\sum_{j<3}\pd_j^2-\pd_t) \Ga  (x-z',s)  u_i(z',0,\tfrac 12-s)\,dz'ds
\\
& =\sum_{j<3} \int_0^{\frac 12} \int_\Si  2 \pd_j \Ga  (x-z',s)  \pd_{z_j} u_i(z',0,\tfrac 12-s)\,dz'ds
\\
& \quad +  \int_0^{\frac 12} \int_\Si  2  \Ga  (x-z',s) \pd_t u_i(z',0,\tfrac 12-s)\,dz'ds
\\
& \quad - \lim_{\mu\to 0_+}  \bke{ \int_\Si  2  \Ga  (x-z',\tfrac12 -\mu) u_i(z',0,\mu )\,dz -  \int_\Si  2  \Ga  (x-z',\mu) u_i(z',0,\tfrac 12-\mu)\,dz}
\\
&=I_3+I_4+ \lim_{\mu\to 0_+} \bke{ I_{5,\mu}+I_{6,\mu}}.
\end{align*}
Here $I_3$ can be estimated in the same way as $I_1$, and
\eqref{Eq5.30} is valid if $I_1$ is replaced by $I_3$.  For $I_4$,
since $\pd_t u_i = \De u_i$, by estimate \eqref{eq6.13t} for $\nb^2 u$,
\begin{equation}
\label{eq5.34}
|I_4|
 \lec \int_0^{\frac 12}  \int_\Si s^{-\frac 32} (1+\frac{|x-z'|^2}{4s})^{-\frac 32}
(\tfrac 12-s)^{-\frac 12}
(|z'|+\sqrt {\tfrac 12-s})^{-2}\,dz'\,ds.
\end{equation}
We have similar estimate as $I_1$ with the following
difference: we have to use the estimate~(\ref{Eq-Nab}) during the
integration over each subinterval $s\in[0,1/4]$ and
$s\in[1/4,1/2]$, for the second subinterval we
apply~(\ref{Eq-Nab}) with $a=\frac12$, $b=1$, $N=|z'|^2$.

For the boundary terms, the integrand of $I_{5,\mu}$ is bounded by
$ e^{-\frac{|x-z'|^2}{2}} |z'|^{-1}$ and converges to $0$ as
$\mu\to 0_+$ for each $z'\in \Si$. Thus $\lim I_{5,\mu}=0$ by
Lebesgue dominated convergence theorem. For $I_{6,\mu}$,
\begin{equation}
\label{eq5.35}
|I_{6,\mu}|\lec \mu^{-1/2}e^{-\frac{x_3^2}{4\mu}} \int_\Si
\Ga_{\R^2} (x'-z',\mu) \frac 1{\bka{z '}}dz' \lec
\mu^{-1/2}e^{-\frac{x_3^2}{4\mu}} \frac 1{\bka{x'}},
\end{equation}
which converges to $0$ as $\mu\to0_+$ for any $x \in \Om$.

We conclude that, for either $k<3$ or $k=3$, \eqref{Eq5.30} is valid
if $I_1$ is replaced by $I_2$ and hence, for any $0<\de\ll 1$,
\begin{equation}
\label{eq6.27}
(1 +x_3)
 |\pd_k w_i(x,\tfrac 12)|\lec x_3^{-\de}\bka{x}^{-2+2\de},
\quad \forall x\in \Om_+, \ \forall i,k \le 3.
\end{equation}

Combining \eqref{eq6.13} %
and \eqref{eq6.27}, %
we have shown \eqref{Eq5.3} in $\Om_+$. This concludes
the proof of Lemma \ref{th:U0}.  \hfill \qed

\section{Self-similar solutions in the half space}
\label{S6}

In this section we first complete the proof of Theorem \ref{th1.1},
and then give a few comments.

\begin{proof}[Proof of Theorem  \ref{th1.1}]

By Lemma \ref{th:U0}, for those $a$ satisfying the assumptions of
Theorem \ref{th1.1}, $U_0 = e^{-\frac 12A}a$ satisfies \eqref{Eq5.2}
and \eqref{Eq5.3}, and hence Assumption \ref{th3.1} is satisfied. By
Theorem \ref{EiuBD}, there is a solution $V \in H^1_{0,\si}(\R^3_+)$
of the system \eqref{eq3.4}--\eqref{eq3.6}.

Noting $U_0 \in C^\I(\R^3_+)$ by \eqref{Eq5.2}, the system
\eqref{eq3.4}--\eqref{eq3.6} is a perturbation of the stationary
Navier-Stokes system with smooth coefficients.  The regularity
theory for Navier-Stokes system implies that $V \in
C^\I_{loc}(\overline{\R^3_+})$.  The vector field $U= U_0 + V$ is
thus a smooth solution of the Leray equations \eqref{eq1.6} in
$\R^3_+$.

The vector field $u(x,t)$ defined by \eqref{SS.formula}, $ u(x,t)=
\frac 1{\sqrt{2t}}\, U\bke{\frac x{\sqrt{2t}}}$, is thus smooth
and self-similar. Moreover,
\[
v(x,t) = u(x,t) - e^{-tA}a = \frac 1{\sqrt{2t}}\, V \bke{\frac
  x{\sqrt{2t}}}
\]
satisfies $\norm{v(t)}_{L^q(\R^3_+)} =
\norm{V}_{L^q(\R^3_+)}(2t)^{\frac3{2q}-\frac12}$ and $\norm{\nb
  v(t)}_{L^2(\R^3_+)} = \norm{\nb V}_{L^2(\R^3_+)} (2t)^{-1/4}$.  This
finishes the proof of Theorem \ref{th1.1}.
\end{proof}

{\it Remark.}  Let $u_0(x,t)= (e^{-tA}a)(x)= \frac 1{\sqrt{2t}}\, U_0
\bke{\frac x{\sqrt{2t}}}$. We have $u_0(\cdot,t) \to a$ as $t\to 0_+$
in $L^{3,\I}(\R^3_+)$. Indeed, by \eqref{Eq5.2}, $|U_0(x)| \lec
\bka{x}^{-1} \in L^{3,\I} \cap L^q$, $q>3$.  We have
$\norm{u_0(t)}_{L^q(\R^3_+)} =
\norm{U_0}_{L^q(\R^3_+)}(2t)^{\frac3{2q}-\frac12}$, which remains
finite as $t\to 0_+$ only if $q=(3,\I)$, and
\begin{equation}
\label{Eq6.2}
|u_0(x,t)| \lec \frac 1{\sqrt t}  \cdot
\frac 1{ 1+\frac{|x|}{\sqrt t}}
 = \frac 1{\sqrt t+|x|}.
\end{equation}
This is consistent with the whole space case $\Om=\R^3$.

For the difference $V(x)$, we only have its $L^q(\R^3_+)$ bounds,
and not pointwise bounds as \eqref{Eq1.14} in \cite{JS,Tsai14}.

\section*{Acknowledgments}
The authors are deeply indebted to K.~Pileckas and R.~Russo for
valuable discussions, in particular, they pointed out the way to
simplify the proofs for more strength result.

This work was initiated when M.K.~visited the Center for Advanced
Study in Theoretical Sciences (CASTS) of the National Taiwan
University. The research of Korobkov was partially supported by
the Russian Foundation for Basic Research (project
No.\,14-01-00768-a) and by Dynasty Foundation. The research of
Tsai was partially supported by the Natural Sciences and
Engineering Research Council of Canada grant 261356-13.

Mikhail Korobkov, Sobolev Institute of Mathematics,
Acad. Koptyug pr.~4, and Novosibirsk State University, Pirogova
str., 2, 630090 Novosibirsk, Russia; korob@math.nsc.ru

\medskip

Tai-Peng Tsai, Department of Mathematics, University of British
Columbia, Vancouver, BC V6T 1Z2, Canada; and Center for Advanced Study
in Theoretical Sciences, National Taiwan University, Taipei, Taiwan;
e-mail: ttsai@math.ubc.ca

\end{document}